\documentclass[10pt,a4paper]{amsart}
\usepackage[utf8]{inputenc}
\usepackage[T1]{fontenc}
\usepackage{amsmath}
\usepackage{amsfonts}
\usepackage{amssymb}
\usepackage[english]{babel}
\newtheorem{thm}{Theorem}

\newtheorem{con}{Conjecture}
\newtheorem{exa}{Example}
\newtheorem{rem}{Remark}
\makeatletter
\@namedef{subjclassname@2020}{%
  \textup{2020} Mathematics Subject Classification}
\makeatother

\title{On two conjectures regarding generalized sequence of derangements}
\author[E. Lipka]{Eryk Lipka}

\address{Institute of Mathematics 
\\Pedagogical University of Cracow
\\ul. Podchor\k{a}\.{z}ych 2, 30-348 Krak\'ow
\\Poland\\
}
\email{eryklipka0@gmail.com}

\author[P. Miska]{Piotr Miska}

\address{Institute of Mathematics 
\\Faculty of Mathematics and Computer Science
\\Jagiellonian University in Cracow
\\ul. {\L}ojasiewicza 6, 30-084 Krak\'ow
\\Poland\\
}
\email{piotr.miska@uj.edu.pl}


\begin{document}
\maketitle

\begin{abstract}
    In \cite{miska} the second author studied arithmetic properties of a class of sequences that generalize the sequence of derangements. The aim of the following paper is to disprove two conjectures stated in \cite{miska}. The first conjecture regards the set of prime divisors of their terms. The latter one is devoted to the order of magnitude of considered sequences.
\end{abstract}

\section{Introduction}
We denote the sets of non-negative integers, positive integers, prime numbers, integers and rational numbers by $\mathbb{N}$, $\mathbb{N}_+$, $\mathbb{P}$, $\mathbb{Z}$ and $\mathbb{Q}$, respectively.

If there exists $n_0\in\mathbb N$ such that some property holds for each integer $n\geq n_0$, then we will write, that this property holds for each $n\gg 0$.

We call that a sequence $(a_n)_{n\in\mathbb{N}}$ is an ultimately geometric progression with ratio $b$ if there are $c,b$ belonging to some ring such that $a_n=cb^n$ for each $n\gg 0$.

Having a given sequence $\textbf{a}=(a_n)_{n\in\mathbb{N}}$ we define the set
$$\mathcal P_{\bf a} = \left\{ p \in \mathbb P : \exists_{n \in \mathbb N}\quad p \mid a_n \right\}.$$

A derangement of a set with $n$ elements is a permutation of this set without any fixed points. We denote the number of derangements in $S_n$ by $D_n$. The sequence of numbers of derangements (or the sequence of derangements, for short) $\left(D_n\right)_{n \in \mathbb N}$ satisfies the following recurrence relation
\begin{equation}
    D_0=1,\quad D_n = nD_{n-1} + \left(-1\right)^n,\quad n>0.
\end{equation}
Let us consider a class $\mathcal{R}$ of generalizations of this sequence. For given $f,g,h \in \mathbb Z \left[ X \right ]$ we define the sequence $\mathbf{a}(f,g,h) = (a_n)_{n\in\mathbb{N}}$ by the recurrence
\begin{equation}
    a_0=g(0),\quad a_n = f(n)a_{n-1} + g(n)\left(h(n)\right)^n,\quad n>0.
\end{equation}

Let us notice that the class $\mathcal{R}$ contains many well-known sequences:
\begin{itemize}
\item if $f,h = 1, g = c\in\mathbb Z$, then $(a_n)_{n\in\mathbb N} = (c(n+1))_{n\in\mathbb N}$ is an arithmetic progression;
\item if $f = q\in\mathbb Z, g = c\in\mathbb Z, h = 0$, then $(a_n)_{n\in\mathbb N} = (cq^n)_{n\in\mathbb N}$ is a geometric progression;
\item if $f = 1, g = c\in\mathbb Z, h = q\in\mathbb Z$, then $(a_n)_{n\in\mathbb N} = (\sum_{j=0}^n cq^j)_{n\in\mathbb N}$ is the sequence of partial sums of a geometric progression;
\item if $f = X,g = 1, h = 0$, then $(a_n)_{n\in\mathbb N} = (n!)_{n\in\mathbb N}$ is the sequence of factorials;
\item if $f = 2X+l, l\in\{0,1\}, g = 1, h = 0$, then $(a_n)_{n\in\mathbb N} = ((2n+l)!!)_{n\in\mathbb N}$ is the sequence of double factorials.
\end{itemize}

The second author investigated arithmetic properties of this class of sequences in \cite{miska}, and he stated the following two conjectures.

\begin{con}[Conjecture 5.1. from \cite{miska}]\label{c1}
If $\textbf{a}=(a_n)_{n\in\mathbb{N}}\in\mathcal{R}$ is not a geometric progression, then the set $\mathcal P_{\bf a}$ is infinite.
\end{con}

\begin{con}[Conjecture 5.2. from \cite{miska}]\label{c2}
If $\textbf{a}=(a_n)_{n\in\mathbb{N}}\in\mathcal{R}$ and there are $b,c \in \mathbb Z$ such that $a_n = cb^n$ for $n\gg 0$, then $a_n = cb^n$ for each $n \in \mathbb N$.
\end{con}

\section{Ultimately geometric sequences}

While attempting to prove the mentioned conjectures, we found unrecoverable error in proof, which led us to a rather simple counterexamples.

\begin{thm} Let $f,g,h\in\mathbb{Z}[X]$. If $(a_n)_{n\in\mathbb{N}} = \mathbf{a}(f,g,h)$ and there are $b,c \in \mathbb Z$ such that $a_n = cb^n$ for $n \gg 0$, then $h \equiv b$ or $gh \equiv 0$.
\end{thm}
\begin{proof}
If $c=0$, then $a_n = 0$ for each sufficiently large positive integer $n$. This yields $gh \equiv 0$. Hence, we assume that $c,g,h\neq 0$.

If $h$ is not constant, then we have $|h(n)| > |b|$ for $n \gg 0$. Thus
$$|a_n| = |f(n)cb^{n-1} + g(n)h(n)^n|\ge |b+1|^n > |cb^n|,\quad n\gg 0.$$
Hence, $a_n \neq cb^n$ for sufficiently large $n$. If $h(n)$ is constant and $|h(n)| > |b|$, then again using the same argument we conclude that $|a_n| > |cb^n|$ for $n\gg 0$. This means that $a_n \neq cb^n$ for sufficiently large $n$. We get that $h$ is constant and $|h| \le |b|$.

If $b = 0$, then $h \equiv 0$, so from now on we assume that $b \neq 0$. If $|b| = 1$, then of course $|h| = |b|$. If $|b| > 1$, then we consider the prime decomposition
$$|b|= p_1^{\alpha_1} \cdot p_2^{\alpha_2}\cdot \ldots,\quad p_i \in \mathbb P, \alpha_i \in \mathbb N_+.$$

The divisibility $p_i^{n\alpha_i} | a_n$ combined with the recurrence for ${\bf a}(f,g,h)$ implies $$p_i^{(n-1)\alpha_i} | f(n)cb^{n-1} + g(n)h^n$$ and $p_i^{(n-1)\alpha_i} | g(n)h^n$ for $n\gg 0$.
If $p_i^{\alpha_i} \nmid h$, then we have $p_i^{n-\alpha_i} | g(n)$, which is a contradiction for arbitrary large $n$ as $g$ is a nonzero polynomial. As a result, $p_i^{\alpha_i} | h$ for every $i$, hence $b | h$. Finally, $|h| \ge |b|$, which leads us to equality $|h| = |b|$.

If $h = -b$, then using the recurrence for ${\bf a}(f,g,h)$ and $n\gg 0$ we have the following chain of equalities.
\begin{align*}
cb^n = & f(n)cb^{n-1} + g(n)b^n(-1)^n\\
cb = & f(n)c + g(n)b(-1)^n\\
f(n) = & b +g(n) \tfrac b c (-1)^{n-1}
\end{align*}
If $g$ is not constant, then the value of $f(n)$ changes sign infinitely many times, which is impossible for a polynomial. If $g$ is a nonzero constant then $f(n)$ would attain two values alternately which is also impossible for a polynomial. Finally we get $h = b$.
\end{proof}

A formula for a sequence $(a_n)_{n\in\mathbb{N}} = \mathbf{a}(f,g,h)$ with $gh \equiv 0$ can be rewritten as
\begin{equation}
    a_n = g(0)\prod_{i=1}^{n}f(i), \quad n\in\mathbb{N}.
\end{equation}
If such a sequence is an ultimately geometric progression, then $f$ is constant and the sequence is a geometric progression. However, if $gh \not\equiv 0$, then it is possible that $\textbf{a}(f,g,h)$ is an ultimately geometric progression but not a geometric progression.

\begin{thm} Let $f,g,h\in\mathbb{Z}[X]$. If $(a_n)_{n\in\mathbb{N}} = \mathbf{a}(f,g,h)$, $gh \not \equiv 0$ and there are $b,c \in \mathbb Z$ such that $a_n = cb^n$ for $n \gg 0$, then
$$g =c - \frac{c}{b} f.$$
Moreover, if $n_0 = \min\left\{n\in \mathbb N : f(n) = 0\right\}$, then $a_n = cb^n$ if and only if $n \ge n_0$.
\end{thm}
\begin{proof}
It is easy to see that $bc=0$ would imply $gh \equiv 0$. We already know from previous theorem that $h = b$, so for $n \gg 0$ we have
\begin{align*}
cb^n = & f(n)cb^{n-1} + g(n)b^n\\
cb = & f(n)c + g(n)b\\
g(n) = & c - f(n) \tfrac c b
\end{align*}
If two polynomials are equal for infinitely many values, then they are equal, so $g =c - \tfrac{c}{b} f$.

Now, if $a_n = cb^n$ and $a_{n-1} \neq cb^{n-1}$, then
$$cb^n = a_n = f(n)a_{n-1}+g(n)b^n = f(n)a_{n-1} + cb^n -cb^{n-1}f(n),$$
which implies
$$0 = f(n)\left(a_{n-1} - cb^{n-1}\right).$$
Thus, $f(n)$ must be zero. On the other hand, if $a_{n-1} = cb^{n-1}$, then
$$a_n = f(n)cb^{n-1} + g(n)b^n = f(n)cb^{n-1} + cb^n - f(n)cb^{n-1} = cb^n.$$
If $f(n) = 0$, then
$$a_n = f(n)a_{n-1} + g(n)b^n = f(n)(a_{n-1} - cb^{n-1}) + cb^n = cb^n.$$
From this we can see, that if $a_n=cb^n$ for $n\geq n_0$, then $n_0$ is the least non-negative integer zero of the polynomial $f$.
\end{proof}

\section{Counterexamples}

The following are examples of sequences $\textbf{a}\in\mathcal{R}$ being ultimately geometric progressions but not geometric progressions and such that the set $\mathcal P_{\bf a}$ is finite. The presented sequences are counterexamples for both Conjectures \ref{c1} and \ref{c2}.

\begin{exa}
The following example shows that for each pair $(b,c)$ of integers there exists a sequence $\textbf{a}=(a_n)_{n\in\mathbb{N}}\in\mathcal{R}$ such that it is not a geometric progression but $a_n=cb^n$ for $n\gg 0$. Let $f = b(2-X), g = c(X-1), h = b$ and $(a_n)_{n\in\mathbb{N}} = \mathbf{a}(f,g,h)$. Then, we have 
\begin{align*}
    a_0 =& -c,\\
    a_1 =& b\cdot (-c) + 0\cdot b = -cb,\\
    a_2 =& 0\cdot (-bc) + c\cdot b^2 = cb^2,\\
    a_n =& b(2-n)\cdot cb^{n-1} + c(n-1)\cdot b^{n} = cb^n \text{ for } n > 2.
\end{align*}
So, we see that this sequence is not a geometric progression despite being ultimately geometric. Because $\mathcal P_{\bf a} = \left\{p\in \mathbb P : p \mid bc\right\}$, we thus see that $\mathcal P_{\bf a}$ is finite on condition that $bc\neq 0$.
\end{exa}

\begin{exa}\label{ex2}
This example shows that for each $n_0\in\mathbb{N}$ there exists a sequence $\textbf{a}=(a_n)_{n\in\mathbb{N}}\in\mathcal{R}$ without zero terms such that $(a_n)_{n\geq n_0}$ is a geometric progression but $(a_n)_{n\geq n_0-1}$ is not. Let us take some $b',c'\in\mathbb{Z}$ and put $f=2b'n_0!-2b'\prod_{j=0}^{n_0-1}(X-j)$, $g=-c'n_0!+2c'\prod_{j=0}^{n_0-1}(X-j)$ and $h=b'n_0!$. Then, the terms of the sequence $\textbf{a}(f,g,h)$ are equal to
\begin{align*}
    a_n=& (1-2^{n+1})c'b'^nn_0!^{n+1}, \mbox{ for } n<n_0,\\
    a_n=& c'b'^nn_0!^{n+1}, \mbox{ for } n\geq n_0.
\end{align*}
As we can see, $a_n=cb^n$ for $n\geq n_0$, where $b=b'dn_0!$ and $c=c'dn_0!$. On the other hand $a_n\neq cb^n$ for any $n\in\{0,...,n_0-1\}$. If $bc\neq 0$, then the set $\mathcal{P}_{\bf a}$ is finite.
\end{exa}

\begin{exa}\label{ex3}
This example also shows that for each $n_0\in\mathbb{N}$ there exists a sequence $\textbf{a}=(a_n)_{n\in\mathbb{N}}\in\mathcal{R}$ without zero terms such that $(a_n)_{n\geq n_0}$ is a geometric progression but $(a_n)_{n\geq n_0-1}$ is not. To the opposition of the previous example, the terms of the sequence $\textbf{a}(f,g,h)$ have the same sign if $b>0$. Let us take some $b',c',d\in\mathbb{Z}$ and put $f=b'n_0!-b'\prod_{j=0}^{n_0-1}(X-j)$, $g=c'(d-1)n_0!+c'\prod_{j=0}^{n_0-1}(X-j)$ and $h=b'dn_0!$. Then, the terms of the sequence $\textbf{a}(f,g,h)$ are equal to
\begin{align*}
    a_n=& (d^{n+1}-1)c'(b')^n(n_0!)^{n+1}, \mbox{ for } n<n_0,\\
    a_n=& c'(b')^nd^{n+1}(n_0!)^{n+1}, \mbox{ for } n\geq n_0.
\end{align*}
As we can see, $a_n=cb^n$ for $n\geq n_0$, where $b=b'dn_0!$ and $c=c'dn_0!$. On the other hand $a_n\neq cb^n$ for any $n\in\{0,...,n_0-1\}$. If $bc\neq 0$, then the set $\mathcal{P}_{\bf a}$ is finite.
\end{exa}

\begin{rem}
If we allow to consider polynomials $f,g,h$ belonging to the ring $\mathbb{Q}_{\mathbb{Z}}[X]:=\{P\in\mathbb{Q}[X]: P(\mathbb{Z})\subset\mathbb{Z}\}$, then putting $f=2b-\frac{2b}{n_0!}\prod_{j=0}^{n_0-1}(X-j)$, $g=-c+\frac{2c}{n_0!}'\prod_{j=0}^{n_0-1}(X-j)$ and $h=b$ we obtain for any pair $(b,c)$ of integers and $n_0\in\mathbb{N}$ a sequence of the form $\textbf{a}(f,g,h)=(a_n)_{n\in\mathbb{N}}$ such that $a_n=cb^n$ if and only if $n\geq n_0$. Just as in Example \ref{ex2} we obtain $a_n=(1-2^{n+1})cb^n$ for $n<n_0$.

Similarly we can modify Example \ref{ex3}. Taking $f=b'-\frac{b}{n_0!}'\prod_{j=0}^{n_0-1}(X-j)$, $g=c'(d-1)+\frac{c'}{n_0!}\prod_{j=0}^{n_0-1}(X-j)$ and $h=b'd$, we get
\begin{align*}
    a_n=& (d^{n+1}-1)c'(b')^n, \mbox{ for } n<n_0,\\
    a_n=& cb^n, \mbox{ for } n\geq n_0,
\end{align*}
where $b=b'd$ and $c=c'd$.
\end{rem}

\section{Further study}
Despite the fact that Conjectures \ref{c1} and \ref{c2} turned out to be false, there remains an open problem of classification of the sequences $\textbf{a}\in\mathcal{R}$ such that the set $\mathcal{P}_{\bf a}$ is finite. We know, that $\textbf{a}\in\mathcal{R}$ may be bounded but if this is the case, then $\mathbf{a}$ is an ultimately geometric progression with ratio being $-1$, $0$ or $1$ (see \cite{miska2}). Hence, at this moment, the only known sequences of the form $\textbf{a}\in\mathcal{R}$ with finite set $\mathcal{P}_{\bf a}$ are ultimately geometric progressions. This is why we state the revised version of Conjecture \ref{c1} as follows.

\begin{con}
If $\textbf{a}\in\mathcal{R}$ is not an ultimately geometric progression, then the set $\mathcal P_{\bf a}$ is infinite.
\end{con}

\end{document}